\documentclass{elsarticle}

\usepackage[english]{babel}
\usepackage{dsfont} 
\usepackage{graphicx}
\usepackage{color}
\usepackage[hyperindex,breaklinks]{hyperref}
\usepackage{enumerate}
\usepackage{amssymb,empheq} 
\usepackage{amsthm} 
\usepackage{thmtools}
\usepackage{cleveref}

\allowdisplaybreaks
\newtheorem{thm}{Theorem}
\newtheorem{lem}[thm]{Lemma}
\newtheorem{prop}[thm]{Proposition}
\newtheorem{cor}[thm]{Corollary}

\theoremstyle{definition}

\newtheorem{dfn}[thm]{Definition}
\newtheorem{exm}[thm]{Example}

\theoremstyle{remark}
\newtheorem{rem}[thm]{Remark}

\newcommand{\exmsymbol}{\hfill$\circ$}

\newcommand{\cset}{\mathds{C}}

\newcommand{\nset}{\mathds{N}}

\newcommand{\rset}{\mathds{R}}

\newcommand{\diff}{\mathrm{d}}

\newcommand{\supp}{\mathrm{supp}\,}

\newcommand{\cN}{\mathcal{N}}

\newcommand{\cS}{\mathcal{S}}

\newcommand{\fS}{\mathfrak{S}}

\author{Philipp J.\ di Dio}
\address{Department of Mathematics and Statistics, University of Konstanz, Universit\"atsstra{\ss}e 10, D-78464 Konstanz, Germany}
\address{Zukunftskolleg, Universtity of Konstanz, Universit\"atsstra{\ss}e 10, D-78464 Konstanz, Germany}
\address{philipp.didio@uni-konstanz.de}

\journal{arXiv}

\title{Absolutely continuous representing\\ measures of complex sequences}

\begin{document}

\begin{abstract}
In 1989, A.\ J.\ Duran [\textit{Proc.\ Amer.\ Math.\ Soc.} \textbf{107} (1989), 731--741] showed, that for every complex sequence $(s_\alpha)_{\alpha\in\nset_0^n}$ there exists a Schwartz function $f\in\cS(\rset^n,\cset)$ with $\supp f\subseteq [0,\infty)^n$ such that $s_\alpha = \int x^\alpha\cdot f(x)~\diff x$ for all $\alpha\in\nset_0^n$.
It has been claimed to be a generalization of the result by T.\ Sherman [\textit{Rend.\ Circ.\ Mat.\ Palermo} \textbf{13} (1964), 273--278], that every complex sequences is represented by a complex measure on $[0,\infty)^n$.
In the present work we use the convolution of sequences and measures to show, that Duran's result is a \emph{trivial consequence} of Sherman's result.
We use our easy proof to extend the Schwartz function result and to show the flexibility in choosing very specific functions $f$.
\end{abstract}


\maketitle

Let $n\in\nset$.
The \emph{Schwartz space} $\cS(\rset^n,\cset)$ consists of all smooth functions $f\in C^\infty(\rset^n,\cset)$ such that
\[\left\| x^\alpha\cdot \big(\partial^\beta f\big)(x) \right\|_\infty < \infty\]
for all $\alpha,\beta\in\nset_0^n$.

In 1989, A.\ J.\ Duran \cite{duran89} showed, that for every complex sequence $(s_\alpha)_{\alpha\in\nset_0^n}$ there exists a Schwartz function $f\in\cS(\rset^n,\cset)$ with $\supp f \subseteq [0,\infty)^n$ such that
\begin{equation}\label{eq:moments}
s_\alpha = \int_{\rset^n} x^\alpha\cdot f(x)~\diff x
\end{equation}
for all $\alpha\in\nset_0^n$.
This result has been claimed to be a generalization of T.\ Sherman's result \cite{sherman64}.
R.\ P.\ Boas \cite{boas39a} (for $n=1$) and T.\ Sherman (for all $n\in\nset)$ proved, that any complex sequence $(s_\alpha)_{\alpha\in\nset_0^n}$ can be represented by a complex representing measure $\mu$ with $|\mu|(\rset^n)<\infty$ and $\supp\mu \subseteq [0,\infty)^n$.

\begin{prop}[{\cite[Thm.\ 1]{sherman64}}]\label{prop:signedAtoms}
Let $n\in\nset$ and $s = (s_\alpha)_{\alpha\in\nset_0^n}$ be a complex sequence.
There exists a complex measure $\mu$ with $\supp\mu\subseteq [0,\infty)^n$ such that
\[  |\mu|(\rset^n)<\infty \qquad\text{and}\qquad s_\alpha = \int x^\alpha~\diff\mu(x)\]
for all $\alpha\in\nset_0^n$.
\end{prop}

Our simple proof to reduce Duran's result to Sherman's is based on the convolution of sequences and measures.
For the readers convenience, we give  the basic definition and properties, which will be needed for our proof.

\begin{dfn}\label{dfn:convolution}
Let $n\in\nset$.
Let $s=(s_\alpha)_{\alpha\in\nset_0^n}$ and $t=(t_\alpha)_{\alpha\in\nset_0^n}$ be two complex sequences.
We define the \emph{convolution} $s*t = (u_\alpha)_{\alpha\in\nset_0^n}$ of $s$ and $t$ by
\[u_\alpha := \sum_{\beta\preceq\alpha} \binom{\alpha}{\beta}\cdot s_\beta\cdot t_{\alpha-\beta}.\]
\end{dfn}

The following basic properties of the convolution have been long known, are simple to prove by direct computations from the \Cref{dfn:convolution}, and have been used in a stronger topological context in \cite[Sect.\ 3]{didio24posPresConst}.
Here, $\mu*\nu$ is the \emph{convolution} of the measures $\mu$ and $\nu$, see e.g.\ \cite[Sect.\ 3.9]{bogachevMeasureTheory}.

\begin{lem}\label{lem:group}
Let $n\in\nset$ and set
\[\fS := \left\{s=(s_\alpha)_{\alpha\in\nset_0^n}\in\cset^{\nset_0^n} \,\middle|\, s_0 \neq 0\right\}.\]
Then $(\fS,*)$ is a commutative group.
\end{lem}

\begin{lem}\label{lem:measureConv}
Let $n\in\nset$.
If $s\in\cset^{\nset_0^n}$ is represented by the complex measure $\mu$ and $t\in\cset^{\nset_0^n}$ is represented by the complex measure $\nu$ with $|\mu|(\rset^n),|\nu|(\rset^n)<\infty$, then $s*t$ is represented by $\mu*\nu$ with
\[|\mu*\nu|(\rset^n)<\infty \qquad\text{and}\qquad \supp(\mu*\nu)\subseteq \supp\mu + \supp\nu.\]
\end{lem}

Note, if $\mu$ and $\nu$ are measures, i.e., they are positive, then $\supp(\mu*\nu) = \supp\mu + \supp\nu$.
We now collected all preliminaries to show, that Duran's result is a trivial consequence of Sherman's.
We therefore formulate it as a corollary.

\begin{cor}[{\cite[p.\ 731, Theorem]{duran89}}]\label{cor:duran}
Let $n\in\nset$ and let $s = (s_\alpha)_{\alpha\in\nset_0^n}$ be a complex sequence.
There exists a Schwartz function $f\in\cS(\rset^n,\cset)$ such that
\begin{equation}\label{eq:signedMoments}
\supp f \subseteq [0,\infty)^n \qquad\text{and}\qquad s_\alpha = \int x^\alpha\cdot f(x)~\diff x
\end{equation}
for all $\alpha\in\nset_0^n$.
\end{cor}
\begin{proof}
Let $g\in C^\infty(\rset^n,\rset)\setminus\{0\}$ with $g\geq 0$ and $\supp g \subseteq [0,1]^n$.
For all $\alpha\in\nset_0^n$,
\[t_\alpha := \int x^\alpha\cdot g(x)~\diff x\]
and $t := (t_\alpha)_{\alpha\in\nset_0^n}$.
Since $g\geq 0$ and $g\neq 0$, $t_0>0$.
By \Cref{prop:signedAtoms}, let $\mu$ be a representing measure of the complex sequence $t^{-1}*s$ with $\supp\mu\subseteq [0,\infty)^n$.
Since
\begin{equation}\label{eq:gConvMu}
g*\mu\in\cS(\rset^n,\cset)\qquad\text{with}\qquad \supp (g*\mu)\subseteq [0,\infty)^n
\end{equation}
and by
\begin{equation}\label{eq:simpleDuran}
s \;\overset{\text{Lem.\ \ref{lem:group}}}{=}\; \underbrace{t * (\,\underbrace{t^{-1} * s}_{\text{Prop.\ \ref{prop:signedAtoms}}}\,)}_{\text{Lem.\ \ref{lem:measureConv}}},
\end{equation}
$f:=g*\mu$ fulfills (\ref{eq:signedMoments}).
\end{proof}

Equation (\ref{eq:gConvMu}) is a straightforward calculation, since all moments of $|\mu|$ exist.

In summary, the whole proof of \Cref{cor:duran} and hence \cite{duran89} reduces solely to (\ref{eq:simpleDuran}).
All technical difficulties and details (Fourier transform, coefficients of Schwartz functions, characterization of a certain space $\Omega_0$ of analytic functions) are removed.
Equation (\ref{eq:simpleDuran}) reduces \cite{duran89} to a \emph{trivial consequence} of \cite{sherman64}.

Equation (\ref{eq:simpleDuran}) also reveals the great flexibility, which we have in choosing $f$.
We can extend \Cref{cor:duran} and allow greater flexibility in $g$ and $\mu$.
The flexibility in the representing measure $\mu$ comes from the following.

\begin{prop}[{\cite[Thm.\ 5]{schmud25signed}}]\label{prop:schmud}
Let $n\in\nset$ and let $K\subseteq\rset^n$ be closed.
Then the following are equivalent:
\begin{enumerate}[(i)]
\item For each complex sequence $(s_\alpha)_{\alpha\in\nset_0^n}$ there exists a complex measure $\mu$ with
\[|\mu|(\rset^n)<\infty,\qquad \supp\mu\subseteq K, \qquad\text{and}\qquad s_\alpha = \int x^\alpha~\diff\mu(x)\]
for all $\alpha\in\nset_0^n$.

\item $K$ is Zariski dense and
\[\cN_d(K) := \left\{ p\in\rset[x_1,\dots,x_n] \,\middle|\, \sup_{x\in K} \frac{|p(x)|}{(1+|x|^2)^d} < \infty \right\}\]
is finite dimensional for all $d\in\nset_0$.
\end{enumerate}
\end{prop}

\begin{exm}
The set $K = \nset_0^n$ fulfills condition (ii) in \Cref{prop:schmud}.
\exmsymbol
\end{exm}

Since we can find complex \emph{atomic} representing measures for any complex sequence, we get the following generalization of \Cref{cor:duran}.

\begin{thm}\label{thm:main}
Let $n\in\nset$, let $s = (s_\alpha)_{\alpha\in\nset_0^n}\in\cset^{\nset_0^n}$ be a complex sequence, and let $g:\rset^n\to\rset$ be a measurable function such that
\[\int g(x)~\diff x \neq 0 \qquad\text{and}\qquad \int |x^\alpha\cdot g(x)|~\diff x < \infty\]
for all $\alpha\in\nset_0^n$.
Let $K$ be a countable set of points in $\rset^n$ such that $K$ is Zariski dense and $\dim\cN_d(K) < \infty$ for all $d\in\nset_0$.

Then, for all $y\in K$, there exist coefficients $c_y\in\cset$ such that
\[f(x) := \sum_{y\in K} c_y\cdot g(x - y)\]
is measurable with
\[s_\alpha = \int x^\alpha\cdot f(x)~\diff x \qquad\text{and}\qquad \int |x^\alpha\cdot f(x)|~\diff x < \infty\]
for all $\alpha\in\nset_0^n$.
\end{thm}
\begin{proof}
Verbatim the same proof as the proof of \Cref{cor:duran}, especially equation (\ref{eq:simpleDuran}), but with \Cref{prop:schmud} instead of \Cref{prop:signedAtoms}.
\end{proof}

The regularity of $f = g * \mu$ is inherited from $g$, as is well-known and used with mollifiers in partial differential equations and test functions in distributions.

\begin{exm}
\begin{enumerate}[\bfseries (a)]
\item If $g\in C^\infty(\rset^n,\rset)$ with $\supp g \subseteq [0,1]^n$ in \Cref{thm:main}, then $f = g*\mu \in \cS(\rset^n,\cset)$ with $\supp f \subseteq [0,\infty)^n$, i.e., we regain \Cref{cor:duran}.

\item If $g$ in \Cref{thm:main} is a step function, then $f$ is a step function with $\supp f\subseteq [0,\infty)^n$.
For example, let $a\geq 0$ and let $g$ be the characteristic function of $[a,a+1)^n$.
Then $\supp f \subseteq [a,\infty)^n$.
\exmsymbol
\end{enumerate}
\end{exm}

\begin{rem}
All presented techniques (Lemmas \ref{lem:group} and \ref{lem:measureConv}) were already well-known and on a textbook level in 1964, when T.\ Sherman's result \cite{sherman64} appeared.
\exmsymbol
\end{rem}

\section*{Funding}

The author is supported by the Deutsche Forschungs\-gemein\-schaft DFG with the grant DI-2780/2-1 and his research fellowship at the Zukunfts\-kolleg of the University of Konstanz, funded as part of the Excellence Strategy of the German Federal and State Government.


\providecommand{\bysame}{\leavevmode\hbox to3em{\hrulefill}\thinspace}
\providecommand{\MR}{\relax\ifhmode\unskip\space\fi MR }

\end{document}